\documentclass[a4,11pt]{amsart}

\usepackage[misc]{ifsym}
\usepackage{latexsym,paralist}
\usepackage{amsmath}
\usepackage{amssymb}
\usepackage{amsthm}
\usepackage[mathcal]{eucal}
\usepackage{amsmath, amssymb,enumerate}
\usepackage{mathrsfs}
\usepackage{enumerate}
\usepackage[colorlinks=true, pdfstartview=FitV, linkcolor=blue, citecolor=blue,
 urlcolor=blue]{hyperref}
\usepackage{color}
\allowdisplaybreaks

 \makeatletter \@addtoreset{equation}{section}

\makeatother \makeatletter

\newtheorem{definition}{Definition}
\newtheorem{theorem}{Theorem}
\newtheorem{proposition}{Proposition}
\newtheorem{lemma}{Lemma}
\newtheorem{remark}{Remark}

\newtheorem{ex}{Example}

\newcommand{\R}{\mathbb{R}}

\newcommand{\C}{\mathbb{C}}

\newcommand{\N}{\mathbb{N}}

\newcommand{\eps}{\varepsilon}

\allowdisplaybreaks
\begin{document}
\title[Higher order Schr\"odinger operators]{Higher order Schr\"odinger operators}

\email{fgregorio@unisa.it}
\email{chiara.spina@unisalento.it} 
\email{ctacelli@unisa.it}
\thanks{The authors are members of the Gruppo Nazionale per l’Analisi Matematica, la Probabilità
 le loro Applicazioni (GNAMPA) of the Istituto Nazionale di Alta Matematica (INdAM). This paper is based on work supported by project ``Elliptic and parabolic problems, heat kernel estimates and spectral theory" CUP D53D23005580006, funded by European Union
Next Generation EU within the PRIN 2022 program (D.D. 104 - 02/02/2022 Ministero dell’Università e della Ricerca MUR)}


\maketitle

\centerline{\scshape
Federica Gregorio$^{
1}$, Chiara Spina$^{2}$, Cristian Tacelli$^{1}$}

\medskip

{\footnotesize
 \centerline{$^1$Dipartimento di Matematica, Università degli Studi di Salerno, Fisciano, Italy}
} 

\medskip

{\footnotesize
 \centerline{$^2$Dipartimento di Matematica e Fisica "Ennio De Giorgi", Università del Salento, Lecce, Italy}
}

\begin{abstract}
In this paper we consider  higher order Schr\"odinger operators
$$\mathcal L u=Lu+Vu,$$ where $L$ denotes a fourth order operator and $V\geq 0$  a suitable potential. We initiate our analysis by considering the constant coefficients differential operator $L=\Delta^2$. Subsequently, we extend our results to more general operators $L$ featuring suitable variable coefficients. 
We are interested in domain characterization and generation properties of these operators in $L^p(\mathbb{R}^N)$ for $p \in (1, \infty)$. To address this problems we employ a noncommutative version of the Dore-Venni theorem due to Monniaux and Pr\"uss and we  prove that the $L^p$-realization of $\mathcal L$ is quasi sectorial and, consequently, generates an analytic semigroup. Furthermore, this approach allows for a sharp characterization of the operator's domain as the intersection of the domains of the bilaplacian  and the multiplication operator. The required assumptions  allow to treat potentials that grow at infinity like $|x|^r$ for some $r<4$.
\end{abstract}

\section{Introduction}

In this paper we investigate fourth order Schrödinger-type operators of the form 
$$\mathcal L u= Lu+Vu$$
with suitable positive, possibly unbounded,  potentials $V$. We first analyze the model operator $L=\Delta^2$ and then we extend our results to more general  operators $L$ with variable coefficients. Our focus is the elliptic and parabolic solvability of the associated problems in $L^p$ spaces for $1<p<\infty$, alongside a precise characterization of the operator's domain.

Second order Schrödinger operators of the form $\Delta-V$  with positive and possibly unbounded potentials  have been widely investigated. A landmark contribution is the work of Kato \cite{kato86} on generation results; further results regarding domain characterization under oscillation conditions or reverse Hölder inequalities assumptions on the potential $V$ can be found in \cite{she, aus-ben, met-pru-rha-sch}. These studies highlight that dissipativity, positivity, and domination principles are crucial for obtaining generation results in the second order case.  

However, replacing the laplacian with the bilaplacian introduces significant analytical challenges. While these higher order operators are physically relevant—appearing in models of elasticity \cite{mel}, free boundary problems \cite{ada}, and nonlinear elasticity \cite{ant}—they necessitate a fundamentally different approach due to the failure of maximum principles, positivity, and dissipativity (see \cite{lan-maz}). Semigroup generation in $L^p(\mathbb{R}^N)$ for $(-\Delta)^m + V$ has been established in \cite{dav-hin98b} for potentials in suitable higher order Kato classes (see also \cite{hua-wan-zhe-dua}).

In the present work, we consider potentials that may grow at infinity like $|x|^r$, with $r<4$, which do not belong to Kato classes. To establish generation results and domain characterizations, we employ a Dore-Venni type approach for the sum of two non-commuting closed operators satisfying instead a specific commutator condition. In particular, our main results rely on the framework developed by Monniaux and Prüss in \cite{mon-pru}.
Once verified sectoriality and bounded imaginary powers properties for the bilaplacian and the multiplication operator, the main point consists in the proof of the commutator estimate which is possible by imposing regularity and growth conditions on the derivatives of the potential up to the third order. Specifically, we assume that 
$V\in W^{3,1}_{\mathrm{loc}}(\R^N)  $ is a non-negative function such that there exist constants
$\alpha \in \left[0,\frac34\right)$ and  $c>0$ 
satisfying
\begin{equation}\label{condV}
|D_iV|\leq c V^\alpha , \quad |D_{ij}V|\leq c V^\alpha, \quad |D_{ijk}V|\leq c V^\alpha\quad  \text{a.e. in}\
\R^N, \forall\,1\leq i,j,k\leq N.\end{equation}

 Applying  \cite[Theorem 1 and Corollary 2]{mon-pru}, we demonstrate that the sum operator is quasi-sectorial. Furthermore, by analyzing the sectoriality angle, we prove that $\mathcal{L}$ generates an analytic semigroup. Finally, we provide a precise description of the domain as the intersection of the domains of the principal part and the potential operator seen as a multiplication operator.
 A similar technique has been employed in \cite{kun-lor-mai-rha} to study the $L^p$-theory for Schr\"odinger systems in the second order case. 

 The paper is organized as follows. In Section \ref{prelim} we recall some fundamental definitions and establish the sectoriality and bounded imaginary powers properties for the operators involved.
  Section \ref{constantcoeff} deals with the generation and domain characterization for the model operator $$\mathcal L u=\Delta^2u+Vu.$$ 
  Building on the techniques developed in the constant coefficient case, in Section \ref{noncostcoeff}, we extend our results to more general operators $$Lu = \sum_{|\alpha|, |\beta| \leq 2} D^{\alpha} a_{\alpha, \beta} D^\beta u.$$ We assume $a_{\alpha,\beta} \in BUC(\mathbb{R}^N;\R) \cap W^{2,\infty}(\mathbb{R}^N; \mathbb{R})$ for $|\alpha|=|\beta|=2$, and $a_{\alpha,\beta} \in L^{\infty}(\mathbb{R}^N; \mathbb{R})$ otherwise, alongside a suitable ellipticity condition.

\medskip
\textbf{Notations.} We use standard notations for function spaces. We denote by $L^p(\R^N)$ and $W^{k,p}(\R^N)$ the standard Lebesgue and Sobolev spaces, respectively. $C_c^\infty(\R^N)$ is the space of test functions and $C_b(\R^N)$ is the space of bounded continuous functions. 

\section{Some preliminary results}\label{prelim}

In this section, we establish some preliminary results that will be essential for the subsequent analysis.
We first recall the definitions of sectorial operators and operators  with bounded imaginary powers.
For $0<\varphi\leq \pi$ we denote by $\Sigma_\varphi = \{ z\in \C\setminus\{0\}: |\arg z| < \varphi\}$ the open sector of angle $\varphi$ in the complex plane and set $\Sigma_0=(0,+\infty)$. 

\begin{definition}
A linear operator $A$ on a Banach space $X$ is called sectorial (of angle $\varphi)$ if 
there exists an angle $\varphi \in (0,\pi)$ such that
$\sigma (A)\subset\overline \Sigma_\varphi $
and
$$M_\varphi = \sup_{z \in \overline \Sigma_\varphi^c} \|zR(z,A)\|_{L(X)} < \infty$$
where   $R(z,A)=(z-A)^{-1}$.

The angle of sectoriality ({or the spectral angle}) $\varphi_A$ of a sectorial operator $A$ is defined as the infimum of all $\varphi\in (0,\pi)$ such that
$A$ is sectorial of angle $\varphi$, that is
\begin{eqnarray*}
\varphi_A =\inf\{ \varphi \in [0, \pi) : \sigma(A)\subset \overline\Sigma_{\varphi} \mbox{ and } M_{\varphi} < \infty\}.
\end{eqnarray*}
\end{definition}
It is well known, see e.g. \cite[Theorem II.4.6]{eng-nag}, that a closed, densely defined operator $-A$   generates a bounded analytic semigroup if and only if
$A$ is sectorial with spectral angle $\varphi_A < \frac{\pi}{2}$. The operator $A$ is called quasi-sectorial if $\nu+ A$ is sectorial
for some $\nu > 0$.

If $A$ is sectorial of angle $\varphi$ we have that $\rho(A)\supset \overline \Sigma^c_\varphi=-\Sigma_{\pi-\varphi} $
and then $\rho (-A)\supset \Sigma_{\pi-\varphi}$. Hence, we can also write
$$M_\varphi = \sup_{\lambda \in \Sigma_{\pi-\varphi}}\|\lambda R(\lambda,-A)\|_{L(X)}=
\sup_{\lambda \in  \Sigma_{\pi-\varphi}}\|\lambda (\lambda+A)^{-1}\|_{L(X)} $$
and
\begin{eqnarray*}
\varphi_A =\inf\{ \varphi \in [0, \pi) : \Sigma_{\pi-\varphi} \subset \rho (-A) \mbox{ and } M_{\varphi} < \infty\}.
\end{eqnarray*}

\begin{definition}
A sectorial linear operator $A$ on a Banach space $X$, injective, densely defined and with dense range is said to admit \emph{bounded imaginary powers} (BIP) if $A^{is}$ defines a bounded operator on $X$ for all $s \in \R$ and there exists $\theta\in [0,\pi)$ such that 
\[
|A^{is}|\leq Ke^{\theta |s|},\quad s\in \R
\]
for some constant $K\geq 1$. 
The \emph{power angle} $\theta_A$ of $A$ is defined by 
\begin{eqnarray*}
\theta_A =\inf\{ \omega \geq 0 : \exists\, M : \|A^{is}\| \leq M e^{\omega |s|}\;\, \forall\, s \in \R\}.
\end{eqnarray*}
\end{definition}

\begin{remark}
By  Pr\"uss-Sohr theorem (\cite{pru-soh}) the spectral angle and the power angle satisfy the following relation $\theta_A \geq \varphi_A$. 
\end{remark}

We assume that the potential $V$ satisfies the following assumptions. $V:\R^N\to\R$, $0\leq V\in W^{3,1}_{\mathrm{loc}}(\R^N)  $,   and there exist constants
$\alpha \in \left[0,\frac34\right)$ and  $c>0$ 
such that 
\eqref{condV}
holds.  For $p\in (1,\infty)$ we define the multiplication operator $V_p$ on $L^p(\R^N)$
by setting $D_p(V) = \{u \in L^p(\R^N) : Vu \in L^p(\R^N)\}$ and $(V_pu)(x)=V(x)u(x)$.

After choosing $\omega_1>0$ and $\omega_2>0$ in order that the operators  $\Delta^2+\omega_1 I$ and  $V_p+\omega_2I$ are invertible, we
set $A_p=\Delta^2+\omega_1 I$ with $D_p(A) = W^{4,p}(\R^N)$ and  $B_p=V_p+\omega_2I$. 
The operators $A_p$ and $B_p$ satisfy the following properties.

\begin{proposition}\label{properties}
Let $1<p<\infty$.
\begin{enumerate}
[(a)]
\item The operator $A_p$ is invertible, sectorial and admits bounded imaginary powers. Its power angle is 0. Consequently, for every
$\vartheta >0$ there exists a constant $c$ such that for $s \in \R$ and $\lambda \in \Sigma_{\pi-\vartheta}$ we have
\begin{eqnarray*}
\|(\lambda + A_p)^{-1}\|_{L^p} \leq \frac{c}{1+|\lambda|},\qquad
\|A_p^{is}\|_{L^p} \leq ce^{\vartheta |s|}.
\end{eqnarray*}
\item The operator $B_p$  is invertible and admits bounded imaginary powers. Its power angle is 0. Consequently,
for every
$\vartheta >0$ there exists a constant $c$ such that for $s \in \R$ and $\lambda \in \Sigma_{\pi-\vartheta}$ we have
\begin{eqnarray*}
\|(\lambda + B_p)^{-1}\|_{L^p} \leq \frac{c}{1+|\lambda|},\qquad
\|B_p^{is}\|_{L^p} \leq ce^{\vartheta |s|}.
\end{eqnarray*}
\end{enumerate}
\end{proposition}

\begin{proof}

(a) By 
\cite[Proposition 13.11]{kun-wei} the operator $A_p$ has  bounded $H^\infty$-calculus on $\Sigma_{\vartheta}$ for every
$\vartheta >0$ (here $A_p$ is $(M,\omega_0)-$elliptic with $M=1$ and $\omega_0=0$, see \cite[Section 6]{kun-wei}).   Since for every $s \in \R$ the function $f(z) = z^{is}$ is bounded and holomorphic on that sector,
and $\|f\|_{L^\infty(\Sigma_\vartheta)}\leq e^{|s|\vartheta}$, the boundedness of the imaginary powers follows.

(b) First observe that $B_p$ is closed, with dense domain, dense range, and it is injective since $B=V+\omega_2\geq \omega_2>0$. Moreover $rg(B)\in [\omega_2,\infty)\subset \Sigma_\vartheta$ for every $\vartheta>0$.  Then the thesis follows from \cite[Example 3]{pru-soh}.
\end{proof}

A similar result also holds for a more general fourth order operator with variable coefficients. In particular, consider the operator
\begin{equation*}
    L=\sum_{|\alpha|\leq 2,
          |\beta|\leq 2 }
    D^{\alpha}a_{\alpha, \beta}D^\beta
\end{equation*}
with
\[
\begin{array}{ll}
     a_{\alpha,\beta}\in BUC(\R^N;\R)\cap W^{2,\infty}(\R^N;\R) & \text{ if } |\alpha|=|\beta|=2  \\
     a_{\alpha,\beta}\in L^{\infty} (\R^N;\R) & \text{ if } |\alpha|,|\beta|<2
\end{array}
\]
satisfying the following ellipticity condition: there exists $\nu>0$ such that

\[\left| \sum_{|\alpha|=|\beta|=2}
   a_{\alpha,\beta}(x)\xi^{\alpha+\beta}+r^4e^{i\theta}
   \right|\geq \nu(|\xi|^4+r^4)
\]
for each $x\in \R^N$, $\xi\in \R^N$ and $r\geq 0$. 
As before  we choose $\omega_1>0$  in order that the operator  $L_p+\omega_1 I$ is invertible and
set $A_p=L_p+\omega_1 I$ with $D_p(A) = W^{4,p}(\R^N)$. Then, the following result holds. 

\begin{proposition}\label{properties-gen}
Let $1<p<\infty$.
The operator $A_p$ is invertible, sectorial and admits bounded imaginary powers with power angle $\theta_{A}<\pi/2$.  Consequently, for every
$\vartheta >\theta_{A}$ there exists a constant $c$ such that for $s \in \R$ and $\lambda \in \Sigma_{\pi-\vartheta}$ we have
\begin{eqnarray*}
\|(\lambda + A_p)^{-1}\|_{L^p} \leq \frac{c}{1+|\lambda|},\qquad
\|A_p^{is}\|_{L^p} \leq ce^{\vartheta |s|}.
\end{eqnarray*}
\end{proposition}
\begin{proof}
The thesis follows by \cite[Theorem 6.1]{duo-sim}. 
\end{proof}

For ease of reference, let us  state here \cite[Corollary 2]{mon-pru}.
\begin{proposition}{\cite[Corollary 2]{mon-pru}}\label{cor2MP}
    Let $A$ and $B$ be closed, linear, densely defined operators in $L^p(\R^N), 1<p<\infty$, which are sectorial and admit bounded imaginary powers, and let $A$ be invertible. Assume that $\theta_A+\theta_B<\pi$, and fix angles $\vartheta_A>\theta_A,\vartheta_B>\theta_B$ such that $\vartheta_A+\vartheta_B<\pi$.  If $A$ and $B$ satisfy the following commutator condition 
    \[\|A(\lambda+A)^{-1}\left[A^{-1}(\mu + B)^{-1} - (\mu+B)^{-1}A^{-1}\right]\|_{L^p}\leq\frac{C}{(1+|\lambda|^{1-\gamma})|\mu|^{1+\beta}}\]
    for some $C\geq0,$ $0\leq \gamma<\beta\leq 1$ for all $\lambda\in \Sigma_{\pi-\vartheta_A},\mu\in\Sigma_{\pi-\vartheta_B}$,
    then $A+B$ with domain $D(A)\cap D(B)$ is closed, and there is a number $\nu_0\geq0$ such that $\nu+A+B$ is invertible for every $\nu>\nu_0$, and there is a constant $c>0$ such that 
    \[|(\nu+A+B)^{-1}|\leq\frac{c}{\nu}\quad \text{for all}\ \nu>\nu_0.\]
    In particular, the operator $\nu_0+A+B$ is sectorial.
\end{proposition}

\section{Generation results and domain characterization}\label{constantcoeff}
In this section, we want to apply Proposition \ref{cor2MP} to the model operator $A_p+B_p=\Delta^2+V$, defined on  the domain $D(A_p)\cap D(B_p)=D(\Delta^2)\cap D(V_p)$, to prove that it is closed and quasi-sectorial.

Hence, in our situation, since $\theta_{A_p}=\theta_{B_p}=0$, we need an estimate of the commutator in the following form
\begin{eqnarray*}\label{comm}
\|A_p(\lambda+A_p)^{-1}\left[A_p^{-1}(\mu + V_p)^{-1} - (\mu+V_p)^{-1}A_p^{-1}\right]\|_{L^p}\leq\frac{C}{(1+|\lambda|^{1-\gamma})|\mu|^{1+\beta}}
\end{eqnarray*}
for some $C\geq0,$ $0\leq \gamma<\beta\leq 1$ for all $\lambda,\ \mu\in \Sigma_{\pi-\eps}$, $\eps>0$ arbitrarily  chosen.

We start by some preliminary computations needed to estimate the above commutator.

\begin{lemma}\label{a-priori}
Let $u\in W^{4,p}(\R^N)$, $\lambda\in\Sigma_{\pi}$. Then
$$|\lambda|^\frac{3}{4}\|D u\|_p+|\lambda|^\frac{1}{2}\|D^2 u\|_p+|\lambda|^\frac{1}{4}\|D^3 u\|_p+|\lambda|\| u\|_p\leq C\|(\lambda+\Delta^2)u\|_p$$ for some positive constant $C$.
\end{lemma}
\begin{proof}
The claim follows by the classical interpolation inequality (see \cite[Proposition 1.3.8]{lor-rha-book}) and by the resolvent estimate.
\end{proof}

\begin{lemma}\label{multiplication-op}
Let $0\leq V\in W^{3,1}_{\mathrm{loc}}(\R^N)$ and suppose that there exist
$\alpha \in [0,1)$, $c>0$ such that 
 \eqref{condV} holds.
Let $\omega_2>0$ and set $M=(\mu+\omega_2+V)^{-1}$,  then, for every $\mu\in \Sigma_{\pi}$,  there exists $C>0$ such that 
\begin{align*}
\|M\|_\infty \leq \frac{C}{|\mu|+\omega_2}
\end{align*}
and 
\begin{align*}
\|M D_iV\|_\infty,\  
\|M D_{ij}V\|_\infty,\
\|M D_{ijk}V\|_\infty
    \leq C\frac{1}{(|\mu|+\omega_2)^{1-\alpha}}.
\end{align*}
\end {lemma}
\begin{proof}
For any $\mu \in \Sigma_\pi$, we let  $\mu=|\mu|e^{i\varphi}$
where $-\pi<\varphi<\pi$ and $\cos \varphi\geq \nu_0$ for some $-1<\nu_0\leq 0.$
Setting $B=\omega_2+V$, one can estimate
\begin{align*}
|(\mu+B(x))^{-1}|^2
  &=\frac{1}{|\mu|^2\sin^2\varphi+(|\mu|\cos \varphi+B(x))^2}
  =\frac{1}{|\mu|^2+2|\mu|B(x)\cos \varphi +B^2(x)}\\
& \leq
  \frac{1}{(1+\nu_0)(|\mu|^2+B^2(x))}\leq \frac{2}{1+\nu_0}\frac{1}{(|\mu|+B(x))^2}=\frac{C}{(|\mu|+\omega_2+V(x))^2}\\
\end{align*}
and then
\begin{align*}
|(\mu+B(x))^{-1}D_iV(x)|
  &\leq C \frac{V^{\alpha}(x)}{|\mu|+\omega_2+V(x)}\\
&   \leq C\frac{(|\mu|+\omega_2+V(x))^\alpha}{|\mu|+\omega_2+V(x)}
  =C\frac{1}{(|\mu|+\omega_2+V(x))^{1-\alpha}}
   \\& \leq \frac{C}{(|\mu|+\omega_2)^{1-\alpha}}.
\end{align*}
Similarly
\begin{align*}
&|(\mu+B(x))^{-1}D_{ij}V(x)|\leq C \frac{V^{\alpha}(x)}{|\mu|+\omega_2+V(x)}
\leq \frac{C}{(|\mu|+\omega_2)^{1-\alpha}}.
\end{align*}
Finally
\begin{align*}
&|(\mu+B(x))^{-1}D_{ijk}V(x)|\leq C \frac{V^{\alpha}(x)}{|\mu|+\omega_2+V(x)}
\leq \frac{C}{(|\mu|+\omega_2)^{1-\alpha}}.
\end{align*}
\end{proof}

\begin{lemma}\label{multiplication-op2}
Let $0\leq V\in W^{3,1}_{\mathrm{loc}}(\R^N)$ and $M=(\mu+\omega_2+V)^{-1}$ as before.
Then, for every $\mu\in \Sigma_{\pi}$ we have $M\in W^{3,\infty}(\R^N)$ and, moreover,  there exists $C=C(\mu,\omega_2)$ such that
\[\|M\|_\infty\leq \frac{C}{|\mu|}\]
and
\begin{align*}
&\|D_iM\|_\infty,\,
  \|D_{ij}M\|_\infty,\,
  \| D_{ijk}M\|_\infty
    \leq \frac{C}{|\mu|^{2-\alpha}}.
\end{align*}
\end{lemma}
\begin{proof}
Setting $B=\omega_2+V$, one can compute
\begin{align}\label{eq:der1M}
 D_i(\mu+B)^{-1}&=-(\mu+B)^{-2}D_iV\\
&=-(\mu+B)^{-1}(\mu+B)^{-1}D_iV\nonumber
\end{align}
and then
$$\| D_i(\mu+B)^{-1}\|_\infty
\leq C\frac{1}{|\mu|+\omega_2}\frac{1}{(|\mu|+\omega_2)^{1-\alpha}}\leq \frac{C}{|\mu|^{2-\alpha}}.$$
For the second term  
\begin{align}\label{eq:der2M}
 D_{ij}(\mu+B)^{-1}&=2(\mu+B)^{-3}D_iVD_jV-(\mu+B)^{-2}D_{ij}V\\
&=(\mu+B)^{-1}\left[2(\mu+B)^{-1}D_iV(\mu+B)^{-1}D_jV-(\mu+B)^{-1}D_{ij}V\right]
\nonumber
\end{align}
and then
$$
\| D_{ij}(\mu+B)^{-1}\|_\infty
  \leq \frac{C}{|\mu|+\omega_2}\left(\frac{1}{(|\mu|+\omega_2)^{2(1-\alpha)}}+\frac{1}{(|\mu|+\omega_2)^{(1-\alpha)}}\right)
   \leq \frac{C}{|\mu|^{2-\alpha}}. $$
Finally,
\begin{align}\label{eq:der3M}
& D_{ijk}(\mu+B)^{-1}\nonumber\\
&=-6(\mu+B)^{-4}D_iVD_jVD_kV
    +2(\mu+B)^{-3}D_{ik}VD_jV+2(\mu+B)^{-3}D_{i}VD_{jk}V
    -(\mu+B)^{-2}D_{ijk}V\nonumber\\
& =(\mu+B)^{-1}\left[(\mu+B)^{-1}D_iV(\mu+B)^{-1}D_jV(\mu+B)^{-1}D_kV\right.\nonumber\\
&\quad +2(\mu+B)^{-1}D_{ik}V(\mu+B)^{-1}D_jV+2(\mu+B)^{-1}D_{i}V(\mu+B)^{-1}D_{jk}V
    \nonumber\\
&\quad \left. -(\mu+B)^{-1}D_{ijk}V\right]\nonumber\\
\end{align}
and then
\begin{align*}
\| D_{ijk}(\mu+B)^{-1}\|_\infty
  &\leq \frac{C}{|\mu|+\omega_2}\left(
    \frac{1}{(|\mu|+\omega_2)^{3(1-\alpha)}}
    +\frac{1}{(|\mu|+\omega_2)^{2(1-\alpha)}}
    +\frac{1}{(|\mu|+\omega_2)^{(1-\alpha)}}\right)\\&
   \leq \frac{C}{|\mu|^{2-\alpha}}.
\end{align*}
\end{proof}

\begin{lemma}\label{lm:commutator} Let $M\in W^{3,\infty}(\R^N)$.
Set $$C(\lambda,\mu)=A_p(\lambda+A_p)^{-1}\left[A_p^{-1}(\mu + B_p)^{-1} - (\mu+B_p)^{-1}A_p^{-1}\right].$$
Then for every $f\in L^p(\R^N)$, setting $g=A_p^{-1}f$,  the following equality holds
\begin{align*}
 -C(\lambda, \mu)f
 &={\rm div}(\lambda+ A_p)^{-1}\nabla (\Delta M g)\\
 &\quad+ (\lambda + A_p)^{-1} \bigg[2(\nabla\Delta M)\cdot (\nabla g)\\
 &\quad+(\Delta M)(\Delta g)+4Tr(D^2MD^2g)\\
 &\quad +4(\nabla M)\cdot (\nabla\Delta g) \bigg].
 \end{align*}    
\end{lemma}
\begin{proof}
    Let $f\in L^p(\R^N)$ and set  $g=A_p^{-1}f\in W^{4,p}(\R^N)$. By Lemma \ref{multiplication-op2}, $Mg\in W^{3,p}(\R^N)$. We observe that $(\lambda+A_p)^{-1}=R(\lambda+\omega_1,-\Delta^2)$ commutes with the derivatives up to order four. Then,
    \begin{align*}
    &A_p(\lambda+A_p)^{-1}\left[ A_p^{-1}M_p-M_pA_p^{-1}\right]f
        =A_p(\lambda+A_p)^{-1} A_p^{-1}(Mf)-A_p(\lambda+A_p)^{-1}(MA_p^{-1}f)\\
    &=(\lambda+A_p)^{-1}(Mf)-(\omega_1+\Delta^2)(\lambda+A_p)^{-1}(Mg)\\
    &=(\lambda+A_p)^{-1}(MA_pg)-\omega_1(\lambda+A_p)^{-1}(Mg)-\Delta (\lambda+A_p)^{-1}\Delta(Mg)\\
        &=(\lambda+A_p)^{-1}(MA_pg)-\omega_1(\lambda+A_p)^{-1}(Mg)-\Delta (\lambda+A_p)^{-1}\left(\Delta Mg+2\nabla M\cdot \nabla g+M\Delta g\right)\\
        &=(\lambda+A_p)^{-1}(MA_pg)-\omega_1(\lambda+A_p)^{-1}(Mg)-\Delta (\lambda+A_p)^{-1}\left(\Delta Mg\right)\\
        &\quad
                -(\lambda+A_p)^{-1}\left(2\Delta(\nabla M\cdot \nabla g)+
                \Delta M\Delta g+2\nabla M\cdot \nabla \Delta g+M\Delta^2g\right)\\
            &=(\lambda+A_p)^{-1}(MA_pg)-(\lambda+A_p)^{-1}M(\omega_1g+\Delta^2g)-\Delta (\lambda+A_p)^{-1}\left(\Delta Mg\right)\\
        &\quad
                -(\lambda+A_p)^{-1}\left(2\Delta(\nabla M\cdot \nabla g)+
                \Delta M\Delta g+2\nabla M\cdot \nabla \Delta g\right)\\
                 &=-{\rm div}\nabla (\lambda+A_p)^{-1}\left(\Delta Mg\right)\\
        &\quad
                -(\lambda+A_p)^{-1}\left(2\Delta(\nabla M\cdot \nabla g)+
                \Delta M\Delta g+2\nabla M\cdot \nabla \Delta g\right)\\
        &=-{\rm div} (\lambda+A_p)^{-1}\nabla (\Delta Mg)\\
        &\quad -(\lambda+A_p)^{-1}\left(2\Delta(\nabla M\cdot \nabla g)+
                \Delta M\Delta g+2\nabla M\cdot \nabla \Delta g\right).
    \end{align*}
\end{proof}

\begin{theorem}\label{dom-gen}
Fix $p\in (1,\infty)$. Assume that
$0\leq V\in W^{3,1}_{\mathrm{loc}}(\R^N)$ and $V$ satisfies the assumption \eqref{condV} with $\alpha\in [0,\frac{3}{4})$. 
Then, the operator $A_p+B_p$, defined on the domain
$D(A_p)\cap D(B_p)$, is closed, densely defined and quasi-sectorial of angle $0$. Therefore $-A_p-B_p$  generates an analytic semigroup.
\end{theorem}
\begin{proof}
Let $\lambda, \mu\in \Sigma_\pi$ and $f\in L^p(\R^N)$. Set $M(x)=\frac{1}{\mu+\omega_2+V(x)}$  and consider the associated multiplication operators $M_p$. Observe that $(\mu+B_p)^{-1}=M_p$. By Lemma \ref{multiplication-op2} we have $M\in W^{3,\infty}(\R^N)$.  Consider the commutator 
\begin{eqnarray*}
C(\lambda, \mu )f = A_p(\lambda+A_p)^{-1}\left[A_p^{-1}(\mu + B_p)^{-1} - (\mu+B_p)^{-1}A_p^{-1}\right]f.
\end{eqnarray*}
Set $g=A_p^{-1}f$. By Lemma \ref{lm:commutator},  we have
\begin{align*}
-C(\lambda, \mu)f
&={\rm div}(\lambda+ A_p)^{-1}\nabla (\Delta M g) + (\lambda + A_p)^{-1} \bigg[2(\nabla\Delta M)\cdot (\nabla g)\\
 &\quad +(\Delta M)(\Delta g)+4Tr(D^2MD^2g)
 +4(\nabla M)\cdot (\nabla\Delta g) \bigg].
 \end{align*}
By Lemma \ref{a-priori},
 for  
$\lambda\in \Sigma_\pi-\omega_1$,
the operator  $D_i(A_p+\lambda)^{-1}$ is  bounded in $L^p$ for every $i=1,\dots,N$ and, taking into account that for $\lambda\in  \Sigma_\pi$ we have $|\lambda+\omega_1|\geq C(|\lambda|+\omega_1)$,
\[
\|D_i(A_p+\lambda)^{-1}\|_p\leq \frac{C}{|\lambda+\omega_1| ^{3/4}}\leq \frac{C}{1+|\lambda| ^{3/4}}.
\]
Therefore,
\begin{align*}
\|{\rm div}(\lambda+ A_p)^{-1}\nabla (\Delta M g)\|_p&\leq
\sum_{i=1}^N \|D_i(A_p+\lambda)^{-1}(D_i (\Delta M g))\|_p\\
&  \leq \sum_{i=1}^N \frac{C}{1+|\lambda| ^{3/4}}\|D_i (\Delta M g)\|_p\\&\leq \frac{C}{1+|\lambda| ^{3/4}}(\|D^3M\|_\infty\|g\|_p+\|D^2M\|_\infty\|\nabla g\|_p)\\
&  \leq \frac{C}{1+|\lambda| ^{3/4}} \frac{1}{|\mu|^{2-\alpha}}\|g\|_{W^{1,p}(\R^N)}.
\end{align*}
By Lemma \ref{a-priori} again (applied with $\lambda=\omega_1$), we have $\|D_i(A_p^{-1} f)\|_p=\|D_ig\|_p\leq C\|f\|_p$ 
and then
\[
\|{\rm div}(\lambda+ A_p)^{-1}\nabla (\Delta M g)\|_p\leq \frac{C}{(1+|\lambda| ^{3/4})\, |\mu|^{2-\alpha}}\|f\|_p.
\]
Concerning the term $(\lambda + A_p)^{-1} (\nabla\Delta M)\cdot (\nabla g)$, as before 
\begin{align*}
&\|(\lambda + A_p)^{-1} (\nabla\Delta M)\cdot (\nabla g)\|_p
  \leq 
 C\frac{1}{|\lambda+\omega_1| }\frac{1}{\mu^{2-\alpha}}\|f\|_p\leq C\frac{1}{1+|\lambda| ^{3/4}}\frac{1}{\mu^{2-\alpha}}\|f\|_p.
\end{align*}  
We argue similarly for the remaining terms, for each of them we get $\|(\lambda + A_p)^{-1}\|_p\leq \frac{C}{|\lambda+\omega_1| }$ by  Lemma \ref{a-priori} for  $\lambda\in \Sigma_\pi-\omega_1$,  $\|D^2M\|_\infty,\ \|\nabla M\|_\infty\leq \frac{C}{\mu^{2-\alpha}}$ by Lemma \ref{multiplication-op2} and 
$$\|D_{ij}g\|_p\leq C\|f\|_p,\|D_{ijk}g\|_p\leq C\|f\|_p$$
by Lemma \ref{a-priori} again (applied with $\lambda=\omega_1$).
We finally obtain for $\lambda \in \Sigma_\pi$
\[
\|C(\lambda, \mu)f\|_p
\leq C\frac{1}{1+|\lambda| ^{3/4}}\frac{1}{\mu^{2-\alpha}}\|f\|_p.
\]
 The assumptions of  Proposition \ref{cor2MP}  are satisfied with $0\leq \frac{1}{4}<1-\alpha\leq 1$ and the claim follows.

\end{proof}

We now provide an example of a potential $V$ that satisfies our standing assumptions.

\begin{ex}
Let $-\infty< r<4$, $V(x)=(1+|x|^2)^\frac{r}{2}$. Then $V$ satisfies the assumptions by choosing  $\alpha>\max\{0, \frac{r-1}{r}\}$ if $r>0$ and any $0<\alpha<1$ if $r<0$.  By Theorem \ref{dom-gen}, it follows that the operator $\Delta^2+V$, defined on the domain
$D(\Delta^2)\cap D(V)$, is closed, densely defined and quasi-sectorial of angle $0$. Therefore $-\Delta^2-V$  generates an analytic semigroup.
\end{ex}

\section{More general operators}\label{noncostcoeff}

In this section we extend the above results to more general fourth order operators with suitable coefficients. More precisely, we consider 
\begin{equation*}
    L=\sum_{|\alpha|\leq 2,
          |\beta|\leq 2 }
    D^{\alpha}a_{\alpha, \beta}D^\beta
\end{equation*}
with
\[
\begin{array}{ll}
     a_{\alpha,\beta}\in BUC(\R^N,\R)\cap W^{2,\infty}(\R^N;\R) & \text{ if } |\alpha|=|\beta|=2  \\
     a_{\alpha,\beta}\in L^{\infty} (\R^N;\R) & \text{ if } |\alpha|,|\beta|<2
\end{array}
\]
satisfying the following ellipticity condition: there exists $\nu>0$ such that
\[\left| \sum_{|\alpha|=|\beta|=2}
   a_{\alpha,\beta}(x)\xi^{\alpha+\beta}+r^4e^{i\theta}
   \right|\geq \nu(|\xi|^4+r^4)
\]
for each $x\in \R^N$, $\xi\in \R^N$ and $r\geq 0$. 

We observe that the formal adjoint
is given by 
\begin{equation*}
    L^*=\sum_{
          |\alpha|\leq 2,|\beta|\leq 2 } (-1)^{|\alpha|+|\beta|}\,
    D^{\beta}a_{\alpha, \beta}D^\alpha
\end{equation*}
and satisfies the same ellipticity condition of $L.$

Recall that, as in the previous section  we choose $\omega_1>0$ and $\omega_2>0$ in order that the operators  $L_p+\omega_1 I$ and  $V_p+\omega_2I$ are invertible and
set $A_p=L_p+\omega_1 I$ with $D_p(A) = W^{4,p}(\R^N)$ and  $B_p=V_p+\omega_2I$.

We consider the formal adjoint
$A^*=(L+\omega_1I)^*$ and observe that the coefficients of $-A^{*}$ satisfy the condition of \cite[Theorem 3.2.2]{lun}, then we have that, after a suitable choice of $\omega_1$, the realization in $L^p(\R^N)$ of $A^*$ is sectorial, with dense domain and range and it coincides with $L_p^*+\omega_1I$. Moreover, the sectorial angles of the operators coincide ($\varphi_{A_p}=\varphi_{A^*_p}$) (\cite[Proposition 1.3 (v)]{denk-hieber-purss}).

\begin{lemma}\label{a-priori-gen}
Let $\lambda\in\Sigma_{\pi-\varphi_{A_p}}$.
Then for every $1\leq i\leq N$ the operator 
$(\lambda+A_p)^{-1}D_i$ is a bounded operator from $W^{1,p}(\R^N)\to L^p(\R^N)$ and
\[
\|(\lambda+A_p)^{-1}D_i\|\leq \frac{C}{|\lambda|^{3/4}}
\] for some positive constant $C$.
\end{lemma}
\begin{proof}
Take $\lambda\in \Sigma_{\pi-\varphi_{A_p}}$. Consider the adjoint operator $A_p^*$,  it satisfies the a-priori estimate
\[
\|D_i(\lambda+A^*_{p'})^{-1}g\|_{p'}
\leq \frac{C}{|\lambda|^{3/4}}
\|g\|_{p'}.
\]
Let $f\in W^{1,p}(\R^N)$, $g\in L^{p'}(\R^N)$ 
and set $v=(\lambda+A_{p'}^*)^{-1}g\in W^{4,p'}(\R^N)$, so $v\in D(A_{p'}^{*})$ and $\lambda v+A^*_{p'}v=g$.
Then,
\begin{align*}
\int (\lambda+A_p)^{-1}D_ifg   & =\int (\lambda+A_p)^{-1}D_if(\lambda +A^*_{p'})v
    =\int (\lambda+A_p)(\lambda+A_p)^{-1}D_ifv\\
&=\int D_if v
=\int -fD_iv.
\end{align*}
Therefore,
\begin{align*}
   \left |\int (\lambda+A_p)^{-1}D_ifg  \right|\leq \|f\|_p\|D_iv\|
    \leq \frac{C}{|\lambda|^{3/4}} \|f\|_p\|g\|_{p'}.
\end{align*}
    
\end{proof}


\begin{lemma}\label{approx-multiplication-op}
Let $0\leq V\in W^{3,1}_{\mathrm{loc}}(\R^N)$  and suppose that there exists
$\alpha \in [0,1)$, $c>0$ such that 
 \eqref{condV} holds.
Let $\mu\in \Sigma_{\pi}$ and set $M=\frac{1}{\mu+\omega_2+V}$, then there exist  a constant $C=C(\mu,\omega_2)$ and a sequence
$M_n\in C^\infty(\R^N)$ such that 
$D^hM_n\to D^hM$ a.e. in $\R^N$ 
for every multi-index $h\in \N_0^N$ such that $ |h|\leq 3$. Moreover,
\begin{align*}
\|M_n\|_\infty \leq \frac{C}{|\mu|}
\end{align*}
and 
\begin{align*}
\|D_iM_n\|_\infty,\  \|D_{ij}M_n\|_\infty,\ \|D_{ijk}M_n\|_\infty
    \leq C\frac{1}{|\mu|^{2-\alpha}}.
\end{align*}
\end {lemma}
\begin{proof}
Let $(\rho_n)_{n\in \N}$ be a mollifier sequence and set $V_n=(\rho_n\star V)\in C^\infty(\R^N)$ and $M_n=\frac{1}{\mu+\omega_2+V_n}$. We have that  $V_n\to V$ in $W_{\rm loc}^{j,p}(\R^N)$ for every $1\leq j\leq 3$. By Lemma   \ref{multiplication-op},
$|M_n|\leq \frac{C}{\omega_2}$ and 
$|M|\leq \frac{C}{\omega_2}$ .
Moreover, $V_n\to V$ uniformly on compact sets and  since
\begin{align*}
M(x)-M_n(x)=M(x)M_n(x)(V_n(x)-V(x))
\end{align*}
$M_n\to M$ uniformly on compact set too.
From equations \eqref{eq:der1M},\eqref{eq:der2M} and \eqref{eq:der3M}
applied to $M_n$ we have that $D^hM_n\to D^hM$
pointwise since $M_n\to M$  and  
$D^hV_n\to D^hV$ pointwise.
Finally observe that
\begin{align*}
    |D^hV_n(x)|&=\left|D^h\int \rho_n(y)V(x-y)dy\right|=\left|\int \rho_n(y)D^hV(x-y)dy\right|\\
    &  \leq \int \rho_n(y)|D^hV(x-y)|dy\
        \leq  C\int \rho_n(y)V^\alpha(x-y)dy\\
    &  \leq C\left(\int \rho_n(y)V(x-y)dy\right)^\alpha 
        =CV_n^\alpha(x).
\end{align*}
Then $V_n$ satisfies the assumptions of Lemma \ref{multiplication-op2} and  the estimates of the derivatives hold.

\end{proof}

Let $M,g\in W^{4,p}(\R^N)$.
We can write
\[
A(Mg)-MAg={\rm div} P(M)g+Q(M)g
\]
where
\[
(P(M))_i=\sum_{\begin{array}{c}
         |\beta|= 3  \\
    \end{array}}b_{e_i,\beta}D^{\beta}M
\]
and
\[
Q(M)=\sum_{\begin{array}{c}
         1\leq |\beta|\leq 3  \\
         0\leq |\gamma|\leq 3
    \end{array}}c_{\beta,\gamma}D^{\beta}MD^{\gamma}
\]
for suitable $b_{\alpha,\beta}$ and $c_{\beta,\gamma}$. 

\begin{theorem}\label{dom-gen2}
Fix $p\in (1,\infty)$. Assume that
$0\leq V\in W^{3,1}_{\mathrm{loc}}(\R^N)$ satisfies the assumptions \eqref{condV} with $\alpha\in [0,\frac{3}{4})$. 
Then, the operator $A_p+B_p$, defined on the domain
$D(A_p)\cap D(B_p)$, is closed, densely defined and quasi-sectorial. Therefore $-A_p-B_p$  generates an analytic semigroup.
\end{theorem}
\begin{proof}
Let $\vartheta>\theta_{A}$, $\lambda\in \Sigma_{\pi-\vartheta}$, 
$\mu\in \Sigma_\pi$
and $f\in L^p(\R^N)$. Set $M(x)=\frac{1}{\mu+\omega_2+V(x)}$, $M_n$ as in Lemma \ref{approx-multiplication-op} and consider the associated multiplication operators $M_p$ and $M_{n,p}$. 
Observe that $(\mu+B_p)^{-1}=M_p$. 

Consider the commutator
\begin{align*}
C(\lambda, \mu )f &= A_p(\lambda+A_p)^{-1}\left[A_p^{-1}(\mu + B_p)^{-1} - (\mu+B_p)^{-1}A_p^{-1}\right]f\\
&= A_p(\lambda+A_p)^{-1}\left[A_p^{-1}M_p - M_pA_p^{-1}\right]f
\end{align*}
and     
\begin{eqnarray*}
C_n(\lambda, \mu )f = A_p(\lambda+A_p)^{-1}\left[A_p^{-1}M_{n,p} - M_{n,p}A_p^{-1}\right]f.
\end{eqnarray*}
We first prove that $C_n(\lambda, \mu )f$ converges strongly in $L^p(\R^N)$, then we prove that it converges to $C(\lambda, \mu )f$.
Since $M_n\in W^{4,\infty}(\R^N)$ we can write
\begin{align*}
    C_n(\lambda,\mu)f=-(\lambda+A_p)^{-1}[A_pM_{n,p}-M_{n,p}A_p]A_p^{-1}f.
\end{align*}
Let $g=A_p^{-1}f\in D(A)=W^{4,p}(\R^N)$.
We have
\[
A(M_ng)-M_nAg={\rm div} P(M_n)g+Q(M_n)g
\]
then
\begin{align*}
& -C_{n}(\lambda, \mu)f
 =  (\lambda + A_p)^{-1} \bigg[
 {\rm div} P(M_n)g+Q(M_n)g\bigg].\\
  \end{align*}
We observe that by Lemma \ref{approx-multiplication-op}
it holds
\[
\|P(M_n)g\|_p\leq \frac{C}{|\mu|^{2-\alpha}}\|g\|_{W^{4,p}(\R^N)}
\]
and
\[
\|Q(M_n)g\|_p\leq \frac{C}{|\mu|^{2-\alpha}}\|g\|_{W^{4,p}(\R^N)}.
\]
Moreover, by Lemma \ref{approx-multiplication-op} again, 
$P(M_n)g\to P(M)g$ and
$Q(M_n)g\to Q(M)g$ strongly in $L^p(\R^N)$.
By Lemma \ref{a-priori-gen}  we have
\[
\|S\|=\|(A_p+\lambda)^{-1}{\rm div}\|_p\leq C\frac{1}{|\lambda| ^{3/4}}.
\]
Then we can write
\begin{align*}
& -C_{n}(\lambda, \mu)f
 =  
 S P(M_n)g+(\lambda+A_p)^{-1}Q(M_n)g.\\
  \end{align*}
By the a-priori estimates, we have $\|D_ig\|_p\leq C\|f\|_p,\|D_{ij}g\|_p\leq C\|f\|_p,\|D_{ijk}g\|_p\leq C\|f\|_p$ for a constant $C$ depending on $\omega_1$.
Finally,
\[
\|C_n(\lambda,\mu)f\|\leq C\left(
\frac{1}{|\lambda|^{3/4}}+\frac{1}{1+|\lambda|}
\right)\frac{1}{|\mu|^{2-\alpha}}\|f\|_p
\leq C\frac{1}{1+|\lambda| ^{3/4}}\frac{1}{|\mu|^{2-\alpha}}\|f\|_p
\]
and, moreover,
$C_n(\lambda,\mu)f$ converges strongly in $L^{p}(\R^N)$.
Now it remains to prove that $C_n(\lambda,\mu)f$ 
\[
C_n(\lambda, \mu )f = A_p(\lambda+A_p)^{-1}\left[A_p^{-1}M_{n,p} - M_{n,p}A_p^{-1}\right]f 
\]
converges to $C(\lambda,\mu)f$ in $L^p(\R^N)$.
Set
\[
h_n=(\lambda+A_p)^{-1}\left[A_p^{-1}M_{n,p} - M_{n,p}A_p^{-1}\right]f.
\]
Since the sequence of functions $M_n$ converges pointwisely to $M$ and it is uniformly bounded, we have that $M_{n,p}f$ converges to $M_pf$,
 $M_{n,p}A^{-1}_pf$ converges to $M_pA^{-1}_pf$ and $h_n$ converges strongly in $L^p(\R^N)$ to
$h=(\lambda+A_p)^{-1}\left[A_p^{-1}M_p - M_pA_p^{-1}\right]f$. Then, since
$-C_n(\lambda,\mu)f=A_ph_n$ converges in $L^p(\R^N)$ and  $A_p$ is closed, we have that $h\in D_p(A)$ and
$
C_n(\lambda,\mu)f\to -A_ph=C(\lambda,\mu)f.
$
Moreover we also have
\begin{align*}
& -C(\lambda, \mu)f
 =  
 S P(M)g+(\lambda+A_p)^{-1}Q(M)g.\\
  \end{align*}
Observe that we also obtained
that the commutator is bounded by
\[
\|C(\lambda, \mu)\|\leq C\frac{1}{1+|\lambda| ^{1-1/4}}\frac{1}{\mu^{1+\beta}}
\]
where $\beta = 1-\alpha$.
The hypothesis of Proposition \ref{cor2MP}   are satisfied with  $0\leq \frac{1}{4}<1-\alpha\leq 1$ and the claim follows.
\end{proof}

%
\bibliographystyle{amsplain}

\bibliography{bibfile}

\end{document}